\documentclass[10pt, reqno]{amsart}
\usepackage{graphicx, amssymb, amsmath, amsthm, color, slashed, cite}
\numberwithin{equation}{section}

\usepackage{hyperref}

\let\Re=\undefined\DeclareMathOperator*{\Re}{Re}
\let\Im=\undefined\DeclareMathOperator*{\Im}{Im}

\newcommand{\R}{\mathbb{R}}
\newcommand{\C}{\mathbb{C}}

\newcommand{\eps}{\varepsilon}

\newtheorem{theorem}{Theorem}[section]

\newtheorem{lemma}[theorem]{Lemma}

\newtheorem{proposition}[theorem]{Proposition}

\theoremstyle{definition}
\newtheorem{definition}[theorem]{Definition}
\newtheorem{remark}[theorem]{Remark}

\theoremstyle{remark}

\newcommand{\qtq}[1]{\quad\text{#1}\quad}

\begin{document}

\title[Inhomogeneous NLS]{Scattering for the non-radial\\ energy-critical inhomogeneous NLS}

\author[Guzm\'an]{Carlos M. Guzm\'an}
\address{Department of Mathematics, UFF, Brazil}
\email{carlos.guz.j@gmail.com}
\author[Murphy]{Jason Murphy} 
\address{Department of Mathematics \& Statistics, Missouri S\&T, USA}
\email{jason.murphy@mst.edu}

\begin{abstract} We prove scattering below the ground state threshold for an energy-critical inhomogeneous nonlinear Schr\"odinger equation in three space dimensions.  In particular, we extend results of Cho, Hong, and Lee \cite{ChoHongLee, ChoLee} from the radial to the non-radial setting.
\end{abstract}

\maketitle

\section{Introduction}

We study the following inhomogeneous nonlinear Schr\"odinger equation (NLS) in three space dimensions:
\begin{equation}\label{nls}
\begin{cases}
i\partial_t u + \Delta u + |x|^{-1}|u|^2 u = 0, \\
u|_{t=0}=u_0\in \dot H^1(\R^3).
\end{cases}
\end{equation}
This is a special case of a more general class of nonlinear Schr\"odinger equations of the form
\begin{equation}\label{inls}
i\partial_t u + \Delta u \pm |x|^{-b}|u|^p u = 0
\end{equation}
for some $b,p>0$.  This class of equations arises, for example, in nonlinear optical systems with spatially dependent interactions (see e.g. \cite{Belmonte}), and has been the subject of a great deal of recent mathematical investigation.  Of particular interest (from the mathematical perspective) have been the well-posedness theory, existence and stability of solitary waves, and the scattering theory \cite{Farah, GENSTU, Guzman, LeeSeo, CFG20, Campos, ChoHongLee, ChoLee, Dinh1, Dinh3, Dinh4, FG, FG2, MMZ, XZ19}. 

The choice of nonlinearity in \eqref{nls} makes the equation a \emph{focusing, energy-critical} model.  To make this precise, we first observe that \eqref{nls} is the Hamiltonian evolution corresponding to the conserved energy
\[
E(u) = \int_{\R^3} \tfrac12|\nabla u(t,x)|^2 - \tfrac14 |x|^{-1} |u(t,x)|^4\,dx.
\]
The term \emph{focusing} refers to the fact that the nonlinear part of the energy is negative.  In particular, the nonlinearity may balance the underlying linear dispersion (leading to solitary wave solutions) or even dominate (leading to wave collapse).  The term \emph{energy-critical} refers to the fact that the energy is invariant under the scaling symmetry associated to \eqref{nls}, namely,
\[
u(t,x)\mapsto \lambda^{\frac12} u(\lambda^2 t,\lambda x)\qtq{for}\lambda>0. 
\]
More generally, the equation with nonlinearity $|x|^{-b}|u|^p u$ in $d\geq3$ spatial dimensions is energy-critical whenever $p=\tfrac{4-2b}{d-2}$.  We focus here on the case of a cubic nonlinearity in three dimensions, although much of the analysis extends to the more general case.

Local well-posedness for data in $\dot H^1$ follows from the standard critical well-posedness theory, utilizing Strichartz estimates and the contraction mapping principle (see Proposition~\ref{P:LWP} below).  In this work, we consider the questions of global well-posedness and long-time behavior of solutions, particularly that of \emph{scattering}.  We say that a solution \emph{scatters} if 
\begin{equation}\label{scatter}
\exists u_\pm\in \dot{H}^1(\R^3)\qtq{such that}\lim_{t\to\pm\infty}\|u(t)-e^{it\Delta}u_\pm\|_{\dot H^1} = 0,
\end{equation}
where $e^{it\Delta}$ is the linear Schr\"odinger group.  While the local theory yields scattering for sufficiently small initial data, the existence of the \emph{ground state} solution
\[
u(t,x)=Q(x):= (1+\tfrac12|x|)^{-1}
\]
shows that we cannot expect scattering in general.  Instead, we will show that this special solution defines a sharp scattering threshold for \eqref{nls}.  The precise result we will prove is the following.

\begin{theorem}[Scattering below the ground state]\label{T} Suppose $u_0\in\dot H^1(\R^3)$ satisfies
\begin{equation}\label{threshold}
E(u_0)<E(Q) \qtq{and} \|u_0\|_{\dot H^1}<\|Q\|_{\dot H^1}.
\end{equation}
Then the corresponding solution $u$ to \eqref{nls} is global and obeys space-time bounds of the form
\[
\int_\R\int_{\R^3} |u(t,x)|^{10}\,dx\,dt \leq C(E(u_0))
\]
for some function $C:(0,E(Q))\to(0,\infty)$.  Consequently, $u$ scatters in the sense of \eqref{scatter}. 
\end{theorem}

Theorem~\ref{T} represents an extension of the results of \cite{ChoLee, ChoHongLee} from the radial to the non-radial setting.  As a matter of fact, the scenario described in Theorem~\ref{T}, in which the nonlinear ground state defines the sharp scattering threshold, is ubiquitous in the setting of focusing nonlinear dispersive equations.  Restricting just to nonlinear Schr\"odinger equations on $\R^d$, we refer the reader to \cite{AN, DHR, HR, FXC, Guevara, KV, DodsonE, DodsonM, KM, DM, DM2, ADM} for the standard power-type NLS (i.e. \eqref{inls} with $b=0$); to \cite{KMVZZ, KMVZ, Hong, LMM, DinhNRNLS, II} for NLS with external potentials; and to \cite{MMZ, FG, FG2, ChoHongLee, ChoLee, CFGM, XZ19} for the inhomogeneous NLS.  The techniques underlying the most of these works (including \cite{ChoLee, ChoHongLee}) stem from the work of Kenig and Merle on the energy-critical NLS \cite{KM}, in which the authors pioneered a strategy now known as the `Kenig--Merle roadmap'.  The proof of Theorem~\ref{T} will be based on this approach, as well.   In this approach, one reduces the problem of scattering to the problem of precluding the possibility of non-scattering solutions below the ground state threshold that possess certain compactness properties.  This latter step is typically done by utilizing conservation laws and localized virial/Morawetz estimates.  The coercivity in the virial identity follows from the variational characterization of the ground state and the sub-threshold hypothesis, while the compactness is vital for localizing the identity in space.   

Compared to the standard power-type NLS, the analysis of the inhomogeneous NLS presents some new challenges arising from the broken translation symmetry.  In previous works, some of these challenges have been avoided by restricting the analysis to radial (i.e. spherically symmetric) solutions (see e.g. \cite{ChoHongLee, ChoLee, FG, FG2, XZ19}).     This includes the papers \cite{ChoHongLee, ChoLee}, which previously considered the energy-critical problem. On the other hand, the works \cite{MMZ, CFGM} have successfully addressed the non-radial case for a range of \emph{energy-subcritical} problems (i.e. with $p<\tfrac{4-2b}{d-2}$ in \eqref{inls}).  In such problems, one works with data in the inhomogeneous space $H^1$, and the threshold is described in terms of both the mass and the energy of the ground state.  In these works (and in other works in the energy-subcritical setting), the $H^1$ assumption provides some \emph{a priori} control over the very low and very high frequencies of solutions.  Practically speaking, this means that the scaling symmetry plays no significant role in the analysis.  This is no longer the case in the energy-critical setting, and hence (as in \cite{ChoHongLee, ChoLee}) we must incorporate the scaling symmetry into the analysis.  The new challenge of the present work is therefore to address the non-radial problem in the scale-invariant setting. 

The analogue of Theorem~\ref{T} for the standard power-type NLS, namely, scattering below the ground state for the quintic NLS in three dimensions, has been established only in the radial setting (see \cite{KM}).  One way to view the difference between the radial and non-radial cases for that problem is as follows.  The minimal non-scattering solutions mentioned above are parametrized in part by a spatial center $x(t)$.  The usual virial arguments may be used effectively to rule out soliton-type solutions as long as one has some control over the size of $x(t)$.  The case of radial solutions leads to the best-case scenario $x(t)\equiv 0$.  In the general case, one may obtain $|x(t)|=o(t)$ as $t\to\infty$ using an argument based on the conservation of momentum, provided one can show that the mass/momentum of minimal blowup solutions are finite.  In particular, this argument seems only to work in dimensions $d\geq 5$ (see \cite{KV}), which corresponds to the range of dimensions for which the ground state
\[
W(x) = (1+\tfrac{1}{d(d-2)}|x|^2)^{-\frac{d}{2}+1}
\] 
belongs to $L^2$.  While the $d=4$ case (for which $W$ belongs to weak $L^2$) has been treated recently \cite{DodsonE}, the $d=3$ case has remained out of reach so far.  

On the other hand, in the setting of \eqref{nls} one may observe that due to the decaying factor in the nonlinearity, a solution with $|x(t)|\to \infty$ should behave like an approximate solution to the \emph{linear} Schr\"odinger equation and hence decay and scatter.  By making this idea precise (see Proposition~\ref{P:embed} below), we ultimately show that minimal non-scattering solutions must obey $x(t)\equiv 0$.  This effectively puts us in the same situation arising in the \emph{radial} problem (but without the need for a radial assumption).  A somewhat similar situation arises when studying NLS with decaying potentials, in which case the scenario $|x(t)|\to\infty$ effectively reduces the problem to that of the standard NLS, for which the sharp scattering threshold is known (see e.g. \cite{LMM, KMVZZ, KMVZ}).  Indeed, the approach we take here is inspired by these works; the key difference is that in our case, the correct `limiting model' is simply the underlying linear equation.  Not only does this provide some degree of simplification in the analysis, but, importantly, it means that our main result is \emph{not} conditional on the successful resolution of the non-radial problem for the $3d$ quintic NLS. 

As we have indicated above, the proof of Theorem~\ref{T} mostly follows the well-trodden path known as the `Kenig--Merle roadmap'.  In fact, some parts of the argument are essentially the same as what appears in \cite{ChoHongLee} (who studied the same problem in the radial setting).  Thus, while our presentation will be mostly self-contained, we will focus more on the novelties in the present analysis and at times omit or merely sketch certain standard arguments, providing appropriate citations along the way.

The rest of the paper will be organized as follows:
\begin{itemize}
\item In the remainder of the introduction, we will sketch the main ideas of the proof of Theorem~\ref{T} with an emphasis on what is new in this work.  We will then discuss some extensions of Theorem~\ref{T} to some related problems.
\item In Section~\ref{S:preliminaries}, we will collect results related to the local theory and stability theory for \eqref{nls}.  We will also collect some results related to the variational characterization of the ground state. 
\item In Section~\ref{S:CC}, we will discuss some results related to Theorem~\ref{T:reduction} below, in which Theorem~\ref{T} is reduced to the problem of precluding `compact' solutions below the ground state threshold.  This section contains the main new ingredient of the present work, namely, Proposition~\ref{P:embed}. 
\item In Section~\ref{S:virial}, we preclude the possibility of compact sub-threshold solutions.  It suffices to preclude two scenarios, namely, the finite-time blowup scenario (for which we use the conservation of mass) and the `soliton-like' scenario (for which we use the localized virial argument).  
\end{itemize}

\subsection{Sketch of the proof of Theorem~\ref{T}}  

We proceed by contradiction.  The first main step is to show that if Theorem~\ref{T} fails, then one may find a minimal non-scattering solution below the ground state threshold.  As a consequence of minimality, this solution can be shown to possess certain compactness properties.  The precise result we need is the following.

\begin{theorem}[Reduction to compact sub-threshold solutions]\label{T:reduction} Suppose that Theorem~\ref{T} fails.  Then there exists a maximal-lifespan solution $u:[0,T_{\max})\times\R^3\to\C$ such that the following hold:
\begin{itemize}
\item $u$ lies below the ground state threshold, i.e.
\[
E(u)<E(Q)\qtq{and} \sup_{t\in[0,T_{\max})}\|u(t)\|_{\dot H_x^1} < \|Q\|_{\dot H_x^1}.
\]
\item $u$ blows up its scattering norm forward in time, i.e.
\[
\|u\|_{L_{t,x}^{10}([0,T_{\max})\times\R^3)}=\infty.
\]
\item there exists a frequency scale function $N:[0,T_{\max})\to(0,\infty)$ such that
\[
\{N(t)^{-\frac12}u(t,N(t)^{-1}x):t\in[0,T_{\max})\}
\]
is precompact in $\dot H^1$.
\end{itemize}
In addition, we may assume $\inf_{t\in[0,T_{\max})}N(t)\geq 1$. 
\end{theorem}

The detailed construction of such solutions in the setting of the standard energy-critical NLS is presented clearly in the works \cite{KV, Visan}. For the energy-critical inhomogeneous NLS in the radial case, the proof of Theorem~\ref{T:reduction} was given in \cite{ChoHongLee}. As we will explain, addressing the non-radial case requires one main new ingredient, which appears as Proposition~\ref{P:embed} below.  Thus, in what follows, we will only sketch the proof of Theorem~\ref{T:reduction}, emphasizing what is new in the non-radial case.

\begin{proof}[Sketch of the proof of Theorem~\ref{T:reduction}]  In light of the small-data scattering theory, the failure of Theorem~\ref{T} implies the existence of a critical energy $E_c\in(0,E(Q))$, defined as the supremum over energy levels for which the kinetic energy constraint guarantees scattering.  In particular, we can find a sequence of initial data $u_{0,n}$ with corresponding global solutions $u_n$ to \eqref{nls} satisfying the following:
\begin{align}
&E(u_{0,n})\nearrow E_c\qtq{and}\| u_{0,n}\|_{\dot H^1} < \|Q\|_{\dot H^1},\label{un-condition1}\\
&\|u_n\|_{L_{t,x}^{10}(\R_\pm\times\R^3)}\to\infty.\label{un-condition2} 
\end{align}

The key is then to prove that there exist a subsequence in $n$ and scales $\lambda_n\in(0,\infty)$ such that $\lambda_n^{\frac12}u_{0,n}(\lambda_n x)$ converges strongly in $\dot H^1$ to some limit $\phi$.  The solution to \eqref{nls} with initial data $\phi$ is then essentially the solution described in Theorem~\ref{T:reduction} (we return to this point below). 

To demonstrate convergence, we utilize the linear profile decomposition adapted to the $\dot H^1\to L_{t,x}^{10}$ Strichartz estimate (see Proposition~\ref{P:LPD} below).  This allows us to write (up to a subsequence)
\[
u_{0,n}(x) = \sum_{j=1}^J g_n^j[e^{it_n^j\Delta}\phi^j] + w_n^J
\]
with all of the properties stated in Proposition~\ref{P:LPD}.  Here we are using the notation
\[
g_n^jf(x) = (\lambda_n^j)^{-\frac12} f(\tfrac{x-x_n^j}{\lambda_n^j})
\]
for suitable $\lambda_n^j$, $x_n^j$, and we may assume that either $t_n^j\equiv 0$ or $t_n^j\to\pm\infty$, and that either $x_n^j\equiv 0$ or $|x_n^j|\to\infty$.

We therefore need to prove the following: 
\begin{itemize}
\item[(i)] there exists a single profile $\phi$ in the decomposition above; 
\item[(ii)] the space-time translation parameters obey $(t_n,x_n)\equiv (0,0)$; 
\item[(iii)] the remainder obeys $w_n\to 0$ strongly in $\dot H^1$.
\end{itemize}

\emph{Proof of (i).} We prove (i) by contradiction.  Assuming the existence of multiple profiles, we can deduce (from \eqref{energy-decoupling}, \eqref{un-condition1}, and Lemma~\ref{energy-trapping}) that each profile obeys $E(\phi^j)<E_c$ as well as the second condition in \eqref{threshold}.  We then aim to construct a corresponding \emph{nonlinear} profile decomposition of the form
\begin{equation}\label{NPD}
u_n^J = \sum_{j=1}^J v_n^j + e^{it\Delta} w_n^J,
\end{equation}
where each $v_n^j$ is a scattering solution to \eqref{nls} associated to the profile $\phi^j$.  Exploiting the orthogonality of profiles (see \eqref{orthogonality} below), we can deduce that $u_n^J$ defines an approximate solution that obeys global space-time bounds and agrees with $u_n$ at $t=0$.  By stability (Proposition~\ref{P:stab}), this implies space-time bounds for $u_n$ that contradict \eqref{un-condition2}.  

In the case $x_n^j \equiv 0$ and $t_n^j \equiv 0$, we let $v^j$ be the global, scattering solution to \eqref{nls} with initial data $\phi^j$.  If instead $x_n^j\equiv 0$ and $t_n^j\to\pm\infty$, we let $v^j$ be the solution that scatters to $e^{it\Delta}\phi^j$ as $t\to\pm\infty$ (see Proposition~\ref{P:LWP}).  In either of these cases, we then define
\[
v_n^j(t,x) = (\lambda_n^j)^{-\frac12} v^j\left(\tfrac{t}{(\lambda_n^j)^2}+t_n^j,\tfrac{x}{\lambda_n^j}\right).
\]
For the case $\tfrac{|x_n^j|}{\lambda_n^j}\to \infty$, we construct the scattering nonlinear profile $v_n^j$ by appealing to Proposition~\ref{P:embed} below. 

\begin{remark} This final case does not arise in the radial setting and necessitates a new approach in the non-radial case. Indeed, we cannot simply construct the nonlinear solution with data $\phi^j$ and then incorporate the translation into $v_n^j$, as the translation symmetry is broken in \eqref{nls}. Instead, we use an approximation given essentially by solutions to the underlying linear equation.  For the details, see Proposition~\ref{P:embed}. This is the new ingredient needed to extend Theorem~\ref{T:reduction} (and ultimately the main result Theorem~\ref{T}) from the radial to the non-radial case. 
\end{remark}

Having constructed the approximations $u_n^J$ in \eqref{NPD}, it remains to prove 
\begin{equation}\label{NPD-props}
\begin{aligned}
& \limsup_{J\to J^*}\limsup_{n\to\infty}\bigl\{ \|u_n^J\|_{L_t^\infty \dot H_x^1} + \|u_n^J\|_{L_{t,x}^{10}} + \|\nabla u_n^J\|_{L_t^5 L_x^{\frac{30}{11}}}\bigr\}\lesssim 1,\\
& \limsup_{J\to J^*}\limsup_{n\to\infty} \|\nabla[(i\partial_t + \Delta)u_n^J + |x|^{-1}|u_n^J|^2 u_n^J] \|_{N(\R)} = 0. 
\end{aligned}
\end{equation}
Indeed, with \eqref{NPD-props} in hand (along with the agreement of the initial data guaranteed by the linear profile decomposition), we may apply Proposition~\ref{P:stab} to derive bounds for the solutions $u_n$ that contradict \eqref{un-condition2}.  We provide a sketch of the proof of \eqref{NPD-props} in Lemma~\ref{L:decouple} below.  The proof relies essentially on the bounds obeyed by the individual solutions $v_n^j$, the orthogonality of the profiles in the form \eqref{orthogonality}, and the vanishing condition \eqref{vanishing}.

\emph{The proof of (ii)--(iii)}. Once we know that there is a single profile $\phi$ in the decomposition, items (ii)--(iii) quickly follow.  If either $t_n\equiv 0$ or $x_n\equiv 0$ fails, then we can derive bounds for the solutions $u_n$ that contradict \eqref{un-condition2}.  Indeed, if $|\tfrac{x_n}{\lambda_n}|\to\infty$, then Proposition~\ref{P:embed} and Proposition~\ref{P:stab} imply space-time bounds for $u_n$.  If $x_n\equiv 0$ but $t_n\to\infty$, say, then applying Proposition~\ref{P:stab} with the approximate solution $e^{it\Delta}u_{n,0}$ (which has asymptotically vanishing space-time norm) again implies bounds for $u_n$.  The fact that $w_n\to 0$ strongly in $\dot H^1$ follows from energy decoupling, the first condition in \eqref{un-condition1}, and Lemma~\ref{energy-trapping}.  

\smallskip

At this point, we have proven the key convergence statement, and we take $u$ to be the maximal-lifespan solution with initial data given by $\phi$.  That $u$ lies below the ground state threshold follows from the construction and Lemma~\ref{energy-trapping}. The blow-up statement follows from Proposition~\ref{P:stab}.  The compactness statement is derived by applying the convergence argument given above to the sequence $\{u(t_n)\}$ for an arbitrary sequence of times $\{t_n\}$.  Importantly, Proposition~\ref{P:embed} guarantees that profiles with $\tfrac{|x_n|}{\lambda_n}\to\infty$ correspond to \emph{scattering} solutions; it is for this reason that we do not need to contend with a moving spatial center $x(t)$ in the parametrization of our minimal non-scattering solution.  The fact that we can arrange for the frequency scale function to be bounded below, at least on half of the maximal lifespan, follows from the arguments of \cite[Section~4]{KV}. \end{proof}

To complete the proof of Theorem~\ref{T}, it remains to show that the type of solution appearing in the statement of Theorem~\ref{T:reduction} cannot exist.  This is achieved by considering separately the scenarios $T_{\max}<\infty$ (finite-time blowup) and $T_{\max}=\infty$ (which we call the `soliton-like' case).  In the case of finite-time blowup, we can show that the (conserved) mass of the solution is identically zero, which yields a contradiction (see Section~\ref{S:FTBU}).  In the soliton-like case, we instead use the localized virial argument to show that the energy of the solution is identically zero, which again yields a contradiction (see Section~\ref{S:soliton}).  Having precluded both scenarios described in Theorem~\ref{T:reduction}, we complete the proof of Theorem~\ref{T}.

\subsection{Extensions and related problems}  In this section, we discuss a few extensions of Theorem~\ref{T}.  

We first recall from above that the model \eqref{inls} is energy-critical in dimensions $d\geq 3$ whenever $p=\tfrac{4-2b}{d-2}$.  The extension to Theorem~\ref{T} to this more general case should rely primarily on a suitable well-posedness theory for the equation, which becomes more subtle when either $b$ or $d$ increases.  In three dimensions, the authors of \cite{ChoHongLee} addressed the range $0<b<\tfrac43$ in the radial setting (although see Remark~\ref{RonLWP} below).  Utilizing suitable spaces for the local theory together with the new ingredient of this paper (Proposition~\ref{P:embed} below) should allow for the extension of Theorem~\ref{T} to this larger range.  Similarly, the extension to higher dimensions would rest primarily on finding suitable spaces in which to develop the local theory. 

We would also like to remark that the arguments given here suffice to establish global well-posedness and scattering for arbitrary $\dot H^1$-data in the defocusing case, that is:

\begin{theorem}\label{T2} For any $u_0\in \dot H^1(\R^3)$, there exists a unique global solution to
\[
i\partial_tu + \Delta u = |x|^{-1} |u|^2 u
\]
with initial data $u_0$.  Furthermore, this solution obeys global $L_{t,x}^{10}$-bounds and scatters in both time directions. 
\end{theorem}

One can similarly consider the extension of Theorem~\ref{T2} a wider range of energy-critical inhomogeneous nonlinearities. 

\subsection*{Acknowledgements} C.M.G. was partially supported by Coordena\c c\~ao de Aperfei\c coamento de Pessoal de N\'ivel Superior - Brazil (CAPES).

\section{Preliminaries}\label{S:preliminaries}

We use the standard notation for mixed Lebesgue space-time norms and Sobolev spaces.  We write $A\lesssim B$ to denote $A\leq CB$ for some $C>0$. If $A\lesssim B\lesssim A$, we write $A\sim B$.  We write $A\lesssim_E B$ to denote $A\leq CB$ for some $C=C(E)>0$.  In the proof of Proposition~\ref{P:embed} and in Section~\ref{S:FTBU}, we briefly utilize the standard Littlewood--Paley projections $P_N$ and the associated Bernstein estimates, e.g.
\[
\| |\nabla|^s P_N f\|_{L^r} \sim N^s \|P_N f\|_{L^r}.
\]
We also utilize Hardy's inequality 
\[
\| |x|^{-s} f\|_{L^r(\R^3)} \lesssim \| |\nabla|^s f\|_{L^r(\R^3)},\quad 1<r<\tfrac{3}{s}.
\]

We assume familiarity with the standard Strichartz estimates for the linear Schr\"odinger equation and their use in developing the well-posedness and stability theory for the corresponding nonlinear models.  For pedagogical treatments of these topics, we refer the reader to \cite{KVClay, Cazenave}.

\subsection{Local well-posedness and stability}

Well-posedness for \eqref{nls} with data in the critical space $\dot H^1$ follows essentially from the arguments of \cite{CazWei}. We refer the reader to \cite[Section~2.1]{ChoHongLee} for the particular case of energy-critical inhomogeneous NLS.  The result we will need is the following. 

\begin{proposition}[Local well-posedness]\label{P:LWP}
\text{ }
\begin{itemize}
\item For any $u_0\in\dot H^1(\R^3)$ and $t_0\in\R$, there exists $T=T(u_0)$ and a unique solution $u:(t_0-T,t_0+T)\times\R^3\to\C$ to \eqref{nls} with $u(t_0)=u_0$. 
\item There exists $\eta_0>0$ so that if 
\[
\|e^{i(t-t_0)\Delta}u_0\|_{L_{t,x}^{10}((t_0,\infty)\times\R^3)}+ \|\nabla e^{i(t-t_0)\Delta}u_0\|_{L_t^5 L_x^{\frac{30}{11}}((t_0,\infty)\times\R^3)}<\eta 
\]
for some $0<\eta<\eta_0$, then the solution to \eqref{nls} with $u(t_0)=u_0$ is forward global and obeys 
\[
\|u\|_{L_{t,x}^{10}((t_0,\infty)\times\R^3)}\lesssim\eta.
\]
The analogous statement holds backward in time or on all of $\R$.
\item For any $\psi\in\dot H^1(\R^3)$, there exist $T>0$ and a solution $u:(T,\infty)\times\R^3\to\C$ to \eqref{nls} obeying $e^{-it\Delta}u(t)\to \psi$ in $\dot H^1$ as $t\to\infty$.  The analogous statement holds backward in time.
\end{itemize}
\end{proposition}

The proof is based on Strichartz estimates and the contraction mapping principle.  As these arguments are standard by now, let us just show the main nonlinear estimate (which, in particular, demonstrates how to deal with the presence of the inhomogeneous term in the nonlinearity); see also Remark~\ref{remark:scatter} below. 

\begin{lemma}\label{L:NLE} We have the following nonlinear estimate on any space-time slab $I\times\R^3$:
\[
\| \nabla[|x|^{-1}|u|^2 u]\|_{L_t^2 L_x^{\frac65}} \lesssim \|u\|_{L_{t,x}^{10}} \|\nabla u\|_{L_t^5 L_x^{\frac{30}{11}}}^2 . 
\]
\end{lemma}

\begin{proof} Using the product rule, chain rule, H\"older's inequality, and Hardy's inequality, we obtain
\begin{align*}
\|\nabla[|x|^{-1}|u|^2 u]\|_{L_t^2 L_x^{\frac65}} & \lesssim \| |x|^{-1} u\|_{L_t^5 L_x^{\frac{30}{11}}}\|u\|_{L_{t,x}^{10}}\sum_{T\in\{\nabla,|x|^{-1}\}} \|Tu\|_{L_t^5 L_x^{\frac{30}{11}}} \\
& \lesssim \|u\|_{L_{t,x}^{10}} \|\nabla u\|_{L_t^5 L_x^{\frac{30}{11}}}^2. 
\end{align*}
\end{proof}

\begin{remark}\label{RonLWP} The form of this estimate explains why we require both the $L_{t,x}^{10}$-norm and the $L_t^5 \dot H_x^{1,\frac{30}{11}}$-norm in the statement of Proposition~\ref{P:LWP}(ii) above.  In particular, it seems that one cannot close the well-posedness arguments using only smallness in the $L_{t,x}^{10}$-norm (as it seems to be claimed in \cite[Section~2.1]{ChoHongLee}). 
\end{remark}

Next, we record the following stability result for \eqref{nls}.  The proof once again follows standard arguments (see e.g. \cite{KVClay} for the standard NLS or \cite[Section~2.2]{ChoHongLee} for the inhomogeneous NLS).  
\begin{remark}[Notation]\label{NI} In what follows, we write
\[
\|F\|_{N(I)} \leq A
\]
to indicate that there exists a decomposition $F=\sum_{j=1}^J F_j$ so that 
\[
\min\bigl\{\|F_j\|_{L_t^1 L_x^2(I\times\R^3)},\|F_j\|_{L_t^2 L_x^{\frac65}(I\times\R^3)}\bigr\} \leq \tfrac{A}{J} \qtq{for all}j.
\]
\end{remark}

\begin{proposition}[Stability]\label{P:stab} Suppose $\tilde u:I\times\R^3\to\C$ obeys
\[
\|\tilde u\|_{L_t^\infty \dot H_x^1(I\times\R^3)} + \|\tilde u\|_{L_{t,x}^{10}(I\times\R^3)} \leq E<\infty. 
\]
Then there exists $\eps_1 = \eps_1(E)>0$ such that if
\begin{align*}
\| \nabla\bigl\{(i\partial_t+\Delta)\tilde u + |x|^{-1}|\tilde u|^2\tilde u\bigr\} \|_{N(I)}& \leq \eps<\eps_1,\\
 \|e^{i(t-t_0)\Delta}[u_0-\tilde u|_{t=t_0}]\|_{L_{t,x}^{10}(I\times\R^3)}  & \leq \eps<\eps_1, 
\end{align*}
for some $t_0\in I$ and $u_0\in\dot H^1(\R^3)$ with $\|u_0-\tilde u|_{t={t_0}}\|_{\dot H^1}\lesssim_E 1$, then there exists a unique solution $u:I\times\R^3\to\C$ with $u|_{t=t_0}=u_0$, which satisfies
\[
\|u-\tilde u\|_{L_{t,x}^{10}(I\times\R^3)}\lesssim \eps \qtq{and} \|u\|_{L_t^\infty \dot H_x^1(I\times\R^3)}+\|u\|_{L_{t,x}^{10}(I\times\R^3)}\lesssim_E 1.
\]

\end{proposition}

\begin{remark}\label{remark:stab} If we additionally assume that $u_0-\tilde u|_{t=t_0}$ is small in $\dot H^1$, then we can upgrade the estimate on $u-\tilde u$ to hold in $L_t^q \dot H_x^{1,r}$ for any Strichartz admissible pair $(q,r)$. 
\end{remark}

\begin{remark}\label{remark:scatter} Solutions to \eqref{nls} with global $L_t^\infty \dot H_x^1\cap L_{t,x}^{10}$ bounds can be seen to scatter as $t\to\pm\infty$ as follows.  We write $E=\|u\|_{L_t^\infty \dot H_x^1}$ and let $\eps=\eps(E)>0$ be a small parameter to be specified below.  We then split $\R$ into finitely many time intervals such that the $L_{t,x}^{10}$ norm is of size $\eps$ on each interval.  We first prove that the $L_t^5 \dot H_x^{1,\frac{30}{11}}$-norm is controlled by $2CE$ on any such interval, say $I=[t_0,t_1]$.  (This in turn yields global $L_t^5 \dot H_x^{1,\frac{30}{11}}$ control.) Control on each interval follows from a continuity argument using Strichartz, the nonlinear estimate in Lemma~\ref{L:NLE}, and choosing $\eps=\eps(E)$ sufficiently small, viz. 
\[
\|\nabla u\|_{L_t^5 L_x^{\frac{30}{11}}([t_0,\tau]\times\R^3)} \leq CE + C\eps \|\nabla u\|_{L_t^5 L_x^{\frac{30}{11}}([t_0,\tau]\times\R^3)}^2.
\]
for any $\tau\in[t_0,t_1]$. With global space-time bounds in both $L_{t,x}^{10}$ and $L_t^5 \dot H_x^{1,\frac{30}{11}}$, it is not difficult to prove that $\{e^{-it\Delta}u(t)\}$ is Cauchy in $\dot H^1$ as $t\to\pm\infty$ by once again appealing to Strichartz and the nonlinear estimate Lemma~\ref{L:NLE}. \end{remark}

\subsection{Variational analysis}  The ground state
\[
Q(x) = (1+\tfrac12|x|)^{-1}
\]
solves the elliptic equation 
\begin{equation}\label{elliptic}
\Delta Q + |x|^{-1} Q^3 = 0
\end{equation}
and optimizes the embedding estimate
\begin{equation}\label{embedding-est}
\| |x|^{-1}|u|^4\|_{L^1} \leq C_1 \|\nabla u\|_{L^2}^4
\end{equation}
(see e.g. \cite{Yanagida}).

Multiplying \eqref{elliptic} by $Q$ and integrating by parts yields
\[
\| \nabla Q\|_{L^2}^2 = \||x|^{-1}Q^4\|_{L^1} = \tfrac{8\pi}{3},
\]
where the final equality follows from direct calculation.  This leads to the following useful identities:
\begin{equation}\label{ywmtbe}
C_1 = \tfrac{3}{8\pi},\quad E(Q) = \tfrac{2\pi}{3}.
\end{equation}

The following lemma shows that solutions to \eqref{nls} that initially obey the sub-threshold condition \eqref{threshold} continue to do so throughout their lifespan.  Furthermore, the energy is coercive in this regime.

\begin{lemma}[Energy trapping]\label{energy-trapping} Suppose that $u:I\times\R^3\to\C$ is a solution to \eqref{nls} obeying \eqref{threshold}.  Then there exists $\delta>0$ such that
\begin{equation}\label{quant-below}
\sup_{t\in I}\|u(t)\|_{\dot H^1}<(1-\delta)\|Q\|_{\dot H^1}.
\end{equation}
Furthermore,
\begin{equation}\label{energy-sim}
 E(u(t))\sim \|u(t)\|_{\dot H^1}^2\sim \|u_0\|_{\dot H^1}^2\qtq{for all}t\in I.
\end{equation}
\end{lemma}

\begin{proof} We first suppose $f\in \dot H^1$ satisfies
\begin{equation}\label{f-is-below}
E(f) < (1-\delta_0)E(Q) = (1-\delta_0)\tfrac{2\pi}{3} \qtq{and} \|f\|_{\dot H^1}^2 \leq \|Q\|_{\dot H^1}^2 = \tfrac{8\pi}{3}. 
\end{equation}
for some $\delta_0>0$.   By the sharp inequality \eqref{embedding-est} and \eqref{ywmtbe}, we may also write
\begin{equation}\label{energy-lb}
E(f) \geq \tfrac12 \|f\|_{\dot H^1}^2 - \tfrac{3}{32\pi}\|f\|_{\dot H^1}^4. 
\end{equation}
Thus, with $y=\|f\|_{\dot H^1}^2$, we have
\[
y^2-\tfrac{16\pi}{3}y + \tfrac{64\pi^2}{9}  = (y-\tfrac{8\pi}{3})^2 > \tfrac{64\pi^2}{9}\delta_0, \qtq{while} y\leq \tfrac{8\pi}{3}.
\]
It follows that $y< \tfrac{8\pi}{3}-\delta_1$ for some $\delta_1$.  Using this, continuity of the flow in $\dot H^1$, and conservation of energy, we deduce \eqref{quant-below}. 

For \eqref{energy-sim}, it suffices (by the conservation of energy) to show that if $f\in \dot H^1$ is as above, then $E(f)\gtrsim \|f\|_{\dot H^1}^2$.  In fact, using $\|f\|_{\dot H^1}^2\leq\tfrac{8\pi}{3}$ and \eqref{energy-lb}, we have
\[
E(f) \geq \|f\|_{\dot H^1}^2\{\tfrac12 - \tfrac{3}{32\pi}\cdot\tfrac{8\pi}{3}\}\geq \tfrac14\|f\|_{\dot H^1}^2. 
\]
\end{proof}

Finally, the next lemma shows that the sub-threshold assumption in \eqref{threshold} yields the coercivity needed to run the virial argument.  In particular, the quantity appearing in the lemma is precisely the functional that shows up in the time derivative of the virial quantity (see Section~\ref{S:soliton} below).

\begin{lemma}[Coercivity]\label{L:coercive} Suppose that $u:I\times\R^3\to\C$ is a solution to \eqref{nls} obeying \eqref{threshold}.  Then there exists $\delta>0$ such that
\[
\int_{\R^3} |\nabla u(t,x)|^2 - |x|^{-1}|u(t,x)|^4\,dx \geq \delta \int_{\R^3} |\nabla u(t,x)|^2\,dx \qtq{uniformly over}t\in I. 
\]
\end{lemma}

\begin{proof} We suppose $f\in \dot H^1(\R^3)$ satisfies \eqref{f-is-below} as above. Then, by the previous lemma, we have
\[
\|f\|_{\dot H^1}^2 \leq (1-\delta)\|Q\|_{\dot H^1}^2 = (1-\delta)\tfrac{8\pi}{3}\qtq{for some}\delta>0.
\]
By the sharp inequality \eqref{embedding-est}, we therefore obtain
\[
\|f\|_{\dot H^1}^2 - \||x|^{-1} |f|^4\|_{L^1} \geq \|f\|_{\dot H^1}^2\{1-\tfrac{3}{8\pi}\|f\|_{\dot H^1}^2\} \geq \delta \|f\|_{\dot H^1}^2.
\]
In light of Lemma~\ref{energy-trapping}, this yields the desired result. \end{proof}

\section{Reduction to minimal non-scattering solutions}\label{S:CC} 

This section contains several results related to the construction of minimal non-scattering solutions, given as Theorem~\ref{T:reduction} above.  In particular, this section contains the essential new ingredient of this work, namely, Proposition~\ref{P:embed} below. 

We begin with a bit of notation.  For $x_n\in\R^3$ and $\lambda_n\in(0,\infty)$, we define
\[
g_n\phi(x) = \lambda_n^{-\frac12}\phi(\tfrac{x-x_n}{\lambda_n}). 
\]
For later use, we observe that
\begin{equation}\label{gc}
g_n e^{it\Delta} = e^{i\lambda_n^2 t\Delta}g_n.
\end{equation}

The main concentration-compactness tool that we need is the following linear profile decomposition of \cite{Keraani}.  The particular version we need (and its proof) can be found in \cite[Theorem~4.1]{Visan}.	

\begin{proposition}[Linear profile decomposition]\label{P:LPD}
Let $u_n$ be a bounded sequence in $\dot H^1(\R^3)$. Then the following holds up to a subsequence:

There exist $J^*\in\mathbb{N}\cup\{\infty\}$; profiles $\phi^j\in \dot H^1\backslash\{0\}$; scales $\lambda_n^j\in(0,\infty)$; space translation parameters $x_n^j\in\R^3$; time translation parameters $t_n^j$; and remainders $w_n^J$ so that writing
\[
g_n^j f(x) = (\lambda_n^j)^{-\frac12} f(\tfrac{x-x_n^j}{\lambda_n^j}),
\]
we have the following decomposition for $1\leq J\leq J^*$:
\[
u_n = \sum_{j=1}^J g_n^j[e^{it_n^j\Delta}\phi^j] + w_n^J.
\]
This decomposition satisfies the following conditions:
\begin{itemize}
\item Energy decoupling: writing $P(u)=\| |x|^{-1} |u|^4\|_{L^1}$, we have
\begin{equation}\label{energy-decoupling}
\begin{aligned}
&\lim_{n\to\infty} \bigl\{ \|\nabla u_n\|_{L^2}^2 - \sum_{j=1}^J \|\nabla \phi^j \|_{L^2}^2 - \|\nabla w_n^J\|_{L^2}^2\bigr\} = 0, \\
&\lim_{n\to\infty} \bigl\{P(u_n) - \sum_{j=1}^J P(e^{it_n^j\Delta}\phi^j)-P(w_n^J)\bigr\} = 0.
\end{aligned}
\end{equation}
\item Asymptotic vanishing of remainders:
\begin{equation}\label{vanishing}
\limsup_{J\to J^*}\limsup_{n\to\infty} \|e^{it\Delta}w_n^J\|_{L_{t,x}^{10}(\R\times\R^3)} = 0. 
\end{equation}
\item Asymptotic orthogonality of parameters: for $j\neq k$, 
\begin{equation}\label{orthogonality}
\lim_{n\to\infty} \biggl\{\log\bigl[\tfrac{\lambda_n^j}{\lambda_n^k}\bigr] + \tfrac{|x_n^j-x_n^k|^2}{\lambda_n^j\lambda_n^k} + \tfrac{|t_n^j(\lambda_n^j)^2-t_n^k(\lambda_n^k)^2|}{\lambda_n^j\lambda_n^k}\biggr\} = \infty.
\end{equation}
\end{itemize}
In addition, we may assume that either $t_n^j\equiv 0$ or $t_n^j\to\pm\infty$, and that either $x_n^j\equiv 0$ or $|x_n^j|\to\infty$. 
\end{proposition}

\begin{remark} In the usual linear profile decomposition used in the analysis of the $3d$ energy-critical NLS, the potential energy decoupling in \eqref{energy-decoupling} is given in terms of the $L^6$-norm.  However, the same arguments suffice to establish decoupling for the functional appearing in \eqref{energy-decoupling}. 
\end{remark}

The next result is the key result of this paper.  It establishes the existence of scattering solutions to \eqref{nls} associated to initial data living sufficiently far from the origin (relative to the length scale associated to the data).  This result ultimately allows us to extend the construction of minimal blowup solutions from the radial to the non-radial setting; moreover, it guarantees that the compact solutions we ultimately construct remain localized near the origin, which facilitates the use of the localized virial argument. 

Before stating the proposition, let us introduce the following notation:
\[
\|f\|_{\dot S^1(\R)} := \|f\|_{L_{t,x}^{10}(\R\times\R^3)} + \| \nabla f\|_{L_t^{5} L_x^{\frac{30}{11}}(\R\times\R^3)}. 
\]

\begin{proposition}\label{P:embed} Let $\lambda_n \in (0,\infty)$, $x_n\in \R^3$, and $t_n\in\R$ satisfy
\[
\lim_{n\to\infty} \tfrac{|x_n|}{\lambda_n} = \infty \qtq{and} t_n\equiv 0 \qtq{or} t_n\to\pm\infty. 
\]
Let $\phi\in \dot H^1(\R^3)$ and define
\[
\phi_n(x) = g_n [e^{it_n\Delta}\phi](x) = \lambda_n^{-\frac12}[e^{it_n\Delta}\phi](\tfrac{x-x_n}{\lambda_n}). 
\]
Then for all $n$ sufficiently large, there exists a global solution $v_n$ to \eqref{nls} satisfying
\[
v_n(0) = \phi_n \qtq{and}\|v_n\|_{\dot S^1(\R)} \lesssim 1,
\] 
with implicit constant depending only on $\|\phi\|_{\dot H^1}$. 

Furthermore, for any $\eps>0$ there exists $N\in\mathbb{N}$ and $\psi\in C_c^\infty(\R\times\R^3)$ so that for $n\geq N$, we have
\[
\|\lambda_n^{\frac12} v_n(\lambda_n^2(t-t_n),\lambda_n x+x_n)-\psi\|_{\dot S^1(\R)} < \eps.
\]
\end{proposition}

\begin{proof} The proof is based on the construction of suitable approximate solutions to \eqref{nls} with initial condition asymptotically matching $\phi_n$.  Our approximation is based on solutions to the \emph{linear} Schr\"odinger equation. 

To define our approximation, we let $\theta\in(0,1)$ be a small parameter to be determined below and introduce a frequency cutoff $P_n$ and spatial cutoff $\chi_n$ as follows.  First, we let
\begin{equation}\label{def-Pn}
P_n = P_{|\frac{x_n}{\lambda_n}|^{-\theta}\leq\cdot\leq |\frac{x_n}{\lambda_n}|^{\theta}}. 
\end{equation}
Next, we take $\chi_n$ to be a smooth function satisfying
\[
\chi_n(x) = \begin{cases} 1 & |x+\tfrac{x_n}{\lambda_n}| \geq \tfrac12 |\tfrac{x_n}{\lambda_n}| \\ 0 & |x+\tfrac{x_n}{\lambda_n}| < \tfrac14 |\tfrac{x_n}{\lambda_n}|,\end{cases}
\]
with $|\partial^\alpha \chi_n|\lesssim |\tfrac{x_n}{\lambda_n}|^{-|\alpha|}$ for all multiindices $\alpha$.  Observe that $\chi_n\to 1$ pointwise as $n\to\infty$.   

We now define approximations $\tilde v_{n,T}$ as follows: 

\begin{definition}[Approximate solutions] First, let
\[
I_{n,T} := [a_{n,T}^-,a_{n,T}^+]:= [-\lambda_n^2 t_n - \lambda_n^2 T,-\lambda_n^2t_n+\lambda_n^2T]
\]
and for $t\in I_{n,T}$ define
\begin{align*}
\tilde v_{n,T}(t) & = g_n[\chi_n P_n e^{i(\lambda_n^{-2}t+t_n)\Delta}\phi] \\
& = \chi_n(\tfrac{\cdot-x_n}{\lambda_n})e^{i(t+\lambda_n^2 t_n)\Delta} g_n[P_n\phi].
\end{align*}

Next, let
\[
I_{n,T}^+ := (a_{n,T}^+,\infty),\quad I_{n,T}^- = (-\infty,a_{n,T}^-)
\]
and set
\[
\tilde v_{n,T}(t) =\begin{cases} e^{i(t-a_{n,T}^+)\Delta} [\tilde v_{n,T}(a_{n,T}^+)] & t\in I_{n,T}^+, \\ e^{i(t-a_{n,T}^-)\Delta}[\tilde v_{n,T}(a_{n,T}^-)]& t\in I_{n,T}^-.\end{cases}
\]
\end{definition}

We will prove the existence of the solutions $v_n$ by applying the stability result in Proposition~\ref{P:stab}.  This requires that we verify the conditions appearing in the statement of that result.

\textbf{Condition 1.}
\begin{equation}\label{vnt-bounds}
\limsup_{T\to\infty}\limsup_{n\to\infty}\bigl\{\|\tilde v_{n,T}\|_{L_t^\infty \dot H_x^1(\R\times\R^3)} + \|\tilde v_{n,T}\|_{\dot S^1(\R)} \bigr\}\lesssim 1.
\end{equation}

\begin{proof} We estimate separately on $I_{n,T}$ and $I_{n,T}^{\pm}$.

We first estimate on $I_{n,T}$.  Noting that 
\[
\|\chi_n\|_{L^\infty} + \|\nabla[\chi_n]\|_{L^3} \lesssim 1
\]
uniformly in $n$, we may use the product rule, H\"older's inequality, and the Sobolev embedding $\dot H^1\hookrightarrow L^6$ to deduce that $\chi_n:\dot H^1\to \dot H^1$ boundedly (as a multiplication operator). Then the $L_t^\infty \dot H_x^1$ bound on $I_{n,T}$ follows from Strichartz.  The $\dot S^1$ bound on $I_{n,T}$ also follows from Sobolev embedding and Strichartz.  With the $L_t^\infty \dot H_x^1$ bound in place on $I_{n,T}$ (in particular, at the endpoints of this interval), the desired bounds on $I_{n,T}^{\pm}$ then readily follow from Sobolev embedding and Strichartz. 
\end{proof}

\textbf{Condition 2.}
\[
\lim_{T\to\infty}\limsup_{n\to\infty}\|\tilde v_{n,T}(0)-\phi_n\|_{\dot H^1} = 0.
\]

\begin{proof} It suffices to treat the following two cases:
\begin{itemize}
\item[(i)] $t_n\equiv 0$ (so that $0\in I_{n,T}$),
\item[(ii)] $t_n\to+\infty$ (so that $0\in I_{n,T}^+$ for all $n$ sufficiently large).  
\end{itemize}
In particular, the case $t_n\to-\infty$ may be handled similarly to case (ii). 

Case (i). If $t_n\equiv 0$, then 
\[
\| \tilde v_{n,T}(0)-\phi_n\|_{\dot H^1} = \|(\chi_n P_n-1)\phi\|_{\dot H^1} \to 0
\]
as $n\to\infty$, by the dominated convergence theorem. 

(ii) If $t_n\to\infty$, then by \eqref{gc} and by construction, we may write
\begin{align*}
\tilde v_{n,T}(0) & = e^{-ia_{n,T}^+\Delta}g_n[\chi_n P_n e^{i(\lambda_n^{-2}a_{n,T}^++t_n)\Delta}\phi] \\ 
&  = g_n e^{it_n\Delta} [e^{-iT\Delta}\chi_n P_n e^{iT\Delta}\phi].
\end{align*}
Recalling $\phi_n = g_n e^{it_n\Delta}\phi$, we obtain 
\begin{align*}
\| \tilde v_{n,T}(0) - \phi_n \|_{\dot H^1} & = \|[e^{-iT\Delta}\chi_n P_n e^{iT\Delta}-1]\phi\|_{\dot H^1}\\
& = \| [\chi_n P_n -1]e^{iT\Delta}\phi \|_{\dot H^1} \to 0
\end{align*}
as $n\to\infty$, again by the dominated convergence theorem.\end{proof}

\textbf{Condition 3.} Defining
\[
e_{n,T}=(i\partial_t + \Delta)\tilde v_{n,T} + |x|^{-1}|\tilde v_{n,T}|^2 \tilde v_{n,T},
\]
we claim that
\[
\lim_{T\to\infty}\limsup_{n\to\infty}\|\nabla e_{n,T} \|_{N(\R)} = 0,
\]
where $N(\cdot)$ is as in Remark~\ref{NI}. 

\begin{proof} We estimate separately on $I_{n,T}$ and $I_{n,T}^\pm$. 

We first treat the interval $I_{n,T}$.  We write $e_{n,T}=e_{n,T}^{\text{lin}}+e_{n,T}^{\text{nl}}$, with
\begin{align*}
e_{n,T}^{\text{lin}} & = \Delta[\chi_n(\tfrac{x-x_n}{\lambda_n})] e^{i(t+\lambda_n^2 t_n)\Delta}g_n[P_n\phi] \\
& \quad + 2\nabla[\chi_n(\tfrac{x-x_n}{\lambda_n})]\cdot e^{i(t+\lambda_n^2 t_n)\Delta}\nabla[g_n P_n\phi]
\end{align*} 
and 
\[
e_{n,T}^{\text{nl}} = \lambda_n^{-1} g_n\bigl\{ |\lambda_n x+ x_n|^{-1} \chi_n^3 |\Phi_n|^2\Phi_n\bigr\},
\]
where
\[
\Phi_n(t,x) = P_n e^{i(\lambda_n^{-2} t+t_n)\Delta} \phi.
\]

Applying the gradient to $e_{n,T}^{\text{lin}}$, we obtain a sum of terms of the form
\[
\partial^j[\chi_n(\tfrac{x-x_n}{\lambda_n})] \cdot e^{i(t+\lambda_n^2 t_n)\Delta} \partial^{3-j}[g_n P_n \phi]\qtq{for}j\in\{1,2,3\},
\]
where $\partial$ denotes a derivative in $x$.  We estimate such a term in $L_t^1 L_x^{2}(I_{n,T}\times\R^3)$ as follows: by H\"older's inequality, Bernstein's inequality, and the fact that $j\geq 1$, we have
\begin{align*}
\| \partial^j&[\chi_n(\tfrac{x-x_n}{\lambda_n})] \cdot e^{i(t+\lambda_n^2 t_n)\Delta}\partial^{3-j}[g_n P_n\phi]\|_{L_t^1 L_x^{2}(I_{n,T}\times\R^3)} \\
& \lesssim |I_{n,T}|\, \|\partial^j[\chi_n(\tfrac{x-x_n}{\lambda_n})]\|_{L_x^\infty} \|\partial^{3-j}[g_nP_n\phi]\|_{L_t^\infty L_x^2} \\
& \lesssim \lambda_n^2T\cdot\lambda_n^{-j}|\tfrac{x_n}{\lambda_n}|^{-j}\cdot\lambda_n^{j-2}\|\partial^{3-j}P_n\phi\|_{L_t^\infty L_x^2} \\
& \lesssim T|\tfrac{x_n}{\lambda_n}|^{-j}|\tfrac{x_n}{\lambda_n}|^{|2-j|\theta}\to 0\qtq{as}n\to\infty  
\end{align*}
for $\theta$ sufficiently small. 

To estimate $e^{\text{nl}}_{n,T}$ on $I_{n,T}$, we begin by observing that by a change of variables and H\"older's inequality, we have
\[
\|\nabla e_{n,T}^{\text{nl}}\|_{L_t^2 L_x^{\frac65}(I_{n,T}\times\R^3)}  \leq \lambda_nT^{\frac12}\| \nabla\bigl[ |\lambda_n x+ x_n|^{-1} \chi_n^3|\Phi_n|^2\Phi_n\bigr]\|_{L_t^\infty L_x^{\frac65}(I_{n,T}\times\R^3)}.
\]
We next note that 
\[
\|\partial^j\bigl[\lambda_n x+x_n|^{-1}\chi_n^3\bigr]\|_{L_x^\infty} \lesssim |\tfrac{x_n}{\lambda_n}|^{-j}|x_n|^{-1},\quad j\in\{0,1\}. 
\]
Thus, using the product rule, H\"older's inequality, Sobolev embedding, and Bernstein (cf. \eqref{def-Pn}), we estimate
\begin{align*}
\| \nabla e_{n,T}^{\text{nl}}\|_{L_t^2 L_x^{\frac65}(I_{n,T}\times\R^3)} &\lesssim T^{\frac12}|\tfrac{x_n}{\lambda_n}|^{-1}\|\Phi_n\|_{L_t^\infty L_x^3}^2  \sum_{j=0}^1 |\tfrac{x_n}{\lambda_n}|^{-j} \|\partial^{1-j}\Phi_n\|_{L_t^\infty L_x^2} \\
& \lesssim T^{\frac12}|\tfrac{x_n}{\lambda_n}|^{-1+\theta}\|\phi\|_{\dot H^1}^3 \to 0 \qtq{as}n\to\infty
\end{align*}
for $\theta$ sufficiently small. 

We turn to the intervals $I_{n,T}^\pm$.  We only consider $I_{n,T}^+$, as a similar argument handles the remaining interval.  In this case, we have
\[
e_{n,T} = |x|^{-1} |V_{n,T}|^2 V_{n,T},
\]
where 
\[
V_{n,T}= e^{i(t-a_{n,T}^+)\Delta}[\tilde v_{n,T}(a_{n,T}^+)]. 
\]
We then estimate as in Lemma~\ref{L:NLE} and use Strichartz, \eqref{vnt-bounds}, and a change of variables to obtain
\begin{align*}
\| \nabla e_{n,T}\|_{L_t^2 L_x^{\frac65}(I_{n,T}^+\times\R^3)} & \lesssim \| V_{n,T}\|_{L_{t,x}^{10}(I_{n,T}^+\times\R^3)}\|\nabla V_{n,T}\|_{L_t^5 L_x^{\frac{30}{11}}(I_{n,T}^+\times\R^3)}^2 \\
& \lesssim \|\phi\|_{\dot H^1}^2\|e^{it\Delta}[\tilde v_{n,T}(a_{n,T}^+)]\|_{L_{t,x}^{10}(\R_+\times\R^3)}. 
\end{align*}
We now use the definition of $\tilde v_{n,T}$ and \eqref{gc} to write
\[
e^{it\Delta}\tilde v_{n,T}(a_{n,T}^+) = e^{it\Delta} g_n [\chi_n P_n e^{iT\Delta}\phi] = g_n e^{i\lambda_n^{-2}t\Delta}[\chi_n P_n e^{iT\Delta}\phi]. 
\]
Thus by a change of variables, Sobolev embedding, and Strichartz, we have
\begin{align}
\|e^{it\Delta}&[\tilde v_{n,T}(a_{n,T}^+)]\|_{L_{t,x}^{10}(\R_+\times\R^3)}\nonumber \\
& = \|e^{it\Delta}\chi_n P_n e^{iT\Delta}\phi \|_{L_{t,x}^{10}(\R_+\times\R^3)}\nonumber \\
& \lesssim \|\nabla[\chi_n P_n -1]e^{iT\Delta}\phi\|_{L_x^2} + \| e^{it\Delta}\phi\|_{L_{t,x}^{10}([T,\infty)\times\R^3)}.\label{l-e-t}
\end{align}
We now observe that the first term in \eqref{l-e-t} tends to zero as $n\to\infty$ by the dominated convergence theorem, while the second term in \eqref{l-e-t} tends to zero as $T\to\infty$ by Strichartz and the monotone convergence theorem. 
\end{proof}

With conditions 1--3 in place, we may apply Proposition~\ref{P:stab} (see also Remark~\ref{remark:stab}) to deduce the existence of solutions $v_n$ to \eqref{nls} with initial data $\phi_n$; furthermore, these solutions have finite $\dot S^1(\R)$ norm (with a bound depending only on $\|\phi\|_{\dot H^1}$), and we have that
\begin{equation}\label{embed-cc1}
\limsup_{T\to\infty}\lim_{n\to\infty} \|v_n-\tilde v_{n,T}\|_{\dot S^1(\R)} = 0.
\end{equation}

With \eqref{embed-cc1} in place, we turn to the approximation by $C_c^\infty$ functions.  We adapt the arguments from \cite{KMVZZ} and only sketch the proof. We let $\eps>0$.  By \eqref{embed-cc1}, it suffices to show that there exists $\psi\in C_c^\infty(\R\times\R^3)$ so that
\[
\|\lambda_n^{\frac12} \tilde v_{n,T}(\lambda_n^2(t-t_n),\lambda_n x+x_n)-\psi(t,x)\|_{\dot S^1(\R)} \lesssim \eps
\]
for all $n,T$ sufficiently large.  In fact, we will see that $\psi$ will be taken to be an approximation to $e^{it\Delta}\phi$. 

We first observe that for $t\in(-T,T)$, we have
\[
\lambda_n^{\frac12} \tilde v_{n,T}(\lambda_n^2(t-t_n),\lambda_n x+x_n) = \chi_n(x) e^{it\Delta}P_n\phi(x).
\]
We can approximate this in $\dot S^1(\R)$ by $e^{it\Delta}\phi$ for large $n$, which can in turn be well-approximated by a compactly supported function of space-time on $[-T,T]\times\R^3$.  For $t>T$ (the case $t<T$ being treated similarly), we have (using \eqref{gc})
\begin{align*}
\lambda_n^{\frac12}\tilde v_{n,T}(\lambda_n^2(t-t_n),\lambda_n x+x_n) & = g_n^{-1} e^{i\lambda_n^2(t-T)\Delta}g_n \chi_n e^{iT\Delta}P_n \phi \\
& = e^{it\Delta}[e^{-iT\Delta}\chi_n e^{iT\Delta}]P_n \phi.
\end{align*}
Again, we can approximate this in $\dot S^1(\R)$ by $e^{it\Delta}\phi$ for large $n$, which is again well-approximated by a compactly supported function of space-time on $[T,\infty)\times\R^3$. 
\end{proof}

Finally, we prove the following decoupling lemma that was needed in the proof of Theorem~\ref{T:reduction} above. 
 
\begin{lemma}[Decoupling of nonlinear profiles]\label{L:decouple} Let
\[
u_n^J = \sum_{j=1}^J v_n^j + e^{it\Delta} w_n^J
\]
be the nonlinear decomposition appearing in \eqref{NPD}, with all of the properties established in the proof of Theorem~\ref{T:reduction}.  Then the following conditions hold: 
\begin{align}
& \limsup_{J\to J^*}\limsup_{n\to\infty}\bigl\{ \|u_n^J\|_{L_t^\infty \dot H_x^1} + \|u_n^J\|_{L_{t,x}^{10}} + \|\nabla u_n^J\|_{L_t^5 L_x^{\frac{30}{11}}}\bigr\}\lesssim 1, \label{unJbds} \\
& \limsup_{J\to J^*}\limsup_{n\to\infty} \|\nabla[(i\partial_t + \Delta)u_n^J + |x|^{-1}|u_n^J|^2 u_n^J] \|_{N(
\R)} = 0. \label{unJapprox}
\end{align}
\end{lemma}

\begin{proof}  The key ingredient is the following decoupling statement for the nonlinear profiles: 
\begin{equation}\label{np-decouple}
\lim_{n\to\infty} \||x|^{-1} v_n^j v_n^k \|_{L_t^{\frac{10}{3}}L_x^{\frac{15}{7}}(\R\times\R^3)}= 0\qtq{for}j\neq k,
\end{equation}
which follows from approximation by functions in $C_c^\infty(\R\times\R^d)$ and the use of the orthogonality conditions \eqref{orthogonality} (see e.g. \cite{Keraani} or \cite[Lemma~7.3]{Visan}).  With this bound in hand, one can utilize estimates similar to those appearing in Lemma~\ref{L:NLE} to obtain \eqref{unJbds} (see e.g. \cite[(7.11)]{Visan}).  We therefore focus on the proof of \eqref{unJapprox}.  To this end, we write $f(z)=|z|^2 z$ and observe
\begin{align}
(i\partial_t + \Delta)u_n^J + |x|^{-1} f(u_n^J) & = |x|^{-1}[f(u_n^J) - f(u_n^J - e^{it\Delta} w_n^J)]\label{enJ1} \\
& + |x|^{-1}\biggl[f\bigl(\sum_{j=1}^J v_n^j\bigr) - \sum_{j=1}^J f(v_n^j)\biggr].\label{enJ2}
\end{align}
Applying a derivative, we find that it suffices to estimate terms of the following type:
\begin{itemize}
\item[(a)] $|x|^{-1} u_n^J\cdot T u_n^J\cdot e^{it\Delta} w_n^J$,\quad $T\in\{|x|^{-1},\nabla\}$,
\item[(b)] $|x|^{-1}u_n^J\cdot u_n^J\cdot \nabla e^{it\Delta}w_n^J,$
\item[(c)] $ |x|^{-1} v_n^j \cdot v_n^k \cdot T v_n^\ell$,\quad $T\in\{|x|^{-1},\nabla\}$,
\end{itemize}
where in item (c) we have $j\neq k$ and $\ell\in\{1,\dots,J\}$.  The terms of type (c) also involve a constant $C_J$ that grows with $J$.  However, we will shortly see that for each fixed $J$, these terms tend to zero in $N(\R)$ as $n\to\infty$, so that this constant is ultimately harmless.  

Terms of the form (a) are straightforward to estimate.  We estimate in $L_t^2 L_x^{\frac65}$, using the spaces appearing in Lemma~\ref{L:NLE} and the vanishing condition \eqref{vanishing}. Terms of the form (c) are instead handled by appealing to the nonlinear decoupling \eqref{np-decouple}. 

Terms of type (b) pose an additional challenge because the vanishing in \eqref{vanishing} is only given in $L_{t,x}^{10}$ (i.e. without derivative).  We adapt the ideas as presented in \cite[(7.16)]{Visan}. We first observe that by the nonlinear decoupling \eqref{np-decouple} and the vanishing condition \eqref{vanishing}, we may reduce to considering terms of the form
\[
\|  u_n^J\|_{L_t^\infty L_x^6}\|  \sum_{j=1}^J |x|^{-1}v_n^j\cdot  \nabla e^{it\Delta}w_n^J\|_{L_t^2 L_x^{\frac32}}\lesssim \|  \sum_{j=1}^J |x|^{-1}v_n^j\cdot  \nabla e^{it\Delta}w_n^J\|_{L_t^2 L_x^{\frac32}}.
\] 
Using the fact that the entire sum is controlled in $L_t^5 L_x^{\frac{30}{11}}$ (so that the tail of the sum is small), we further reduce to estimating a fixed finite sum, and thereby to estimating a single term of the form
\begin{equation}\label{lclsmth1}
\||x|^{-1} v_n^j \nabla e^{it\Delta} w_n^J\|_{L_t^2 L_x^{\frac32}}.
\end{equation} 
Using approximation of $v^j$ by a compactly supported function of space-time (in the $L_t^5 \dot H^{1,\frac{30}{11}}$-norm) and applying H\"older's inequality, the problem further reduces to showing that
\begin{equation}\label{lclsmth2}
\limsup_{J\to J^*}\limsup_{n\to\infty} \|\nabla e^{it\Delta} w_n^J\|_{L_{t,x}^2(K)}  = 0 \qtq{for any compact}K\subset\R\times\R^3. 
\end{equation}
This finally follows from an interpolation argument using a local smoothing estimate and the vanishing \eqref{vanishing}.  We refer the reader once again to \cite{Visan} (particularly Lemma~2.12 therein) for the details. \end{proof}
\section{Preclusion of compact solutions}\label{S:virial}

Throughout this section, we suppose $u:[0,T_{\max})\times\R^3\to\C$ is the minimal non-scattering solution given by Theorem~\ref{T:reduction}.  In particular, $u$ is parametrized by a frequency scale function $N(t)$ that obeys $\inf_{t\in[0,T_{\max})}N(t)\geq 1$. 

\subsection{Finite-time blowup}\label{S:FTBU} To preclude the finite-time blowup scenario, we utilize the following reduced Duhamel formula, which is a consequence of the compactness properties of $u$.  For the proof, see \cite[Proposition~5.23]{KVClay}.
 
\begin{proposition}[Reduced Duhamel formula]\label{P:RD} For $t\in[0,T_{\max})$, the following holds as a weak limit in $\dot H^1$:
\[
u(t) = i\lim_{T\to T_{\max}} \int_t^T e^{i(t-s)\Delta}[|x|^{-1}|u|^2 u(s)]\,ds.
\]
\end{proposition}

Combining this result with conservation of mass, we can rule out the finite-time blowup scenario.

\begin{proposition}[No finite-time blowup] If $T_{\max}<\infty$, then $u\equiv 0$.
\end{proposition}

\begin{proof}  We suppose that $T_{\max}<\infty$.  Then, using Proposition~\ref{P:RD}, Strichartz estimates, H\"older's inequality, Bernstein's inequality, and Hardy's inequality, we find that for any $t\in[0,T_{\max})$ and any $N>0$, 
\begin{align*}
\|P_N u(t)\|_{L_x^2} & \lesssim \|P_N(|x|^{-1}|u|^2 u)\|_{L_t^1 L_x^2([t,T_{\max})\times\R^3)} \\
& \lesssim N(T_{\max}-t)\| |x|^{-1}|u|^2 u\|_{L_t^\infty L_x^{\frac65}} \\
& \lesssim N(T_{\max}-t)\| u\|_{L_t^\infty L_x^6}^2 \| \nabla u\|_{L_t^\infty L_x^2}. 
\end{align*}
Thus, using Bernstein's inequality for the high frequencies, we deduce
\[
\|u(t)\|_{L_x^2} \lesssim N(T_{\max}-t) + N^{-1}.
\]
for any $t\in[0,T_{\max})$ and $N>0$.  Using conservation of mass, we deduce $\|u\|_{L^2}\equiv 0$ and hence $u\equiv 0$, as desired.
\end{proof}

As $u$ is not identically zero, we conclude that the finite-time blowup scenario is not possible.

\subsection{Soliton-like case}\label{S:soliton} In this section we assume $T_{\max}=\infty$ and derive a contradiction using a localized virial identity.  

\begin{theorem} If $T_{\max}=\infty$, then $u\equiv 0$.
\end{theorem}
 
\begin{proof} For a smooth weight $a$, we define
\[
M_a(t) = 2\Im\int \bar u u_j a_j\,dx,
\]
where we use subscripts to denote partial derivatives and sum repeated indices.  Using a computation using \eqref{nls} and integration by parts, we then have
\begin{equation}\label{virial}
\frac{dM_a}{dt} = \int 4\Re a_{jk}\bar u_j  u_k - |u|^2 a_{jjkk} - |x|^{-1}|u|^4 a_{jj} - |x|^{-3}|u|^4 x_j a_j\,dx. 
\end{equation}
The standard virial identity corresponds to the choice $a(x)=|x|^2$.  However, in this case we cannot guarantee finiteness of $M_a(t)$, as we are working with merely $\dot H^1$ data.  Thus it is essential to localize the weight in space.  That we may do so successfully relies on the compactness of the solution $u(t)$.  In particular, we can establish the following.
\begin{lemma}[Tightness]\label{L:tight} Let $\eps>0$.  Then there exists $R=R(\eps)$ sufficiently large so that
\begin{equation}\label{tight}
\sup_{t\in[0,\infty)} \int_{|x|>R} \bigl\{|\nabla u(t,x)|^2 + |x|^{-1}|u(t,x)|^4 + |x|^{-2}|u(t,x)|^2\bigr\}\,dx<\eps.
\end{equation}
\end{lemma}

\begin{proof} From Theorem~\ref{T:reduction}, we have that $\{N(t)^{-\frac12}u(t,N(t)^{-1}x):t\in[0,\infty)\}$ is tight in $\dot H^1$.  By a change of variables, this implies that for any $\eps>0$, there exists $R=R(\eps)>0$ large enough that
\[
\sup_{t\in[0,\infty)}\int_{|x|>\frac{R}{N(t)}} |\nabla u(t,x)|^2\,dx < \eps. 
\]
Recalling that $\inf_{t\in[0,\infty)}N(t)\geq 1$, we note that we may discard the factor $N(t)$ in the integral above.  This handles the first term in \eqref{tight}.  The other terms may be included (after possibly enlarging $R$) in light of the continuous embeddings $\dot H^1\hookrightarrow L^4(|x|^{-1}\,dx)$ and $\dot H^1 \hookrightarrow L^2(|x|^{-2}\,dx)$, which follow from Hardy's inequality and Sobolev embedding.   \end{proof}

Let us now fix $\eps>0$ and choose $R=R(\eps)$ as in the lemma.  We then choose our weight $a$ such that  
\[
a(x) = \begin{cases} |x|^2 &\text{for } |x|\leq R \\ CR^2 & \text{for } |x|>2R\end{cases}
\]
for some $C>1$.  In the intermediate region, we can impose
\[
|\partial^{\alpha} a| \lesssim R^{2-|\alpha|} \qtq{for} R<|x|\leq 2R
\]
for any multiindex $\alpha$. With this choice of weight, H\"older's inequality, and Lemma~\ref{energy-trapping}, we obtain the upper bound 
\[
\sup_{t\in[0,\infty)} |M_a(t)| \lesssim R^2 \|u\|_{L_t^\infty \dot H_x^1}^2\lesssim R^2 E(u).
\]
On the other hand, the identity \eqref{virial} yields
\begin{align}
\frac{dM_a}{dt} & = 8 \int_{\R^3} |\nabla u|^2 - |x|^{-1}|u|^4\,dx \label{virial-main} \\
& + \mathcal{O}\biggl\{ \int_{|x|>R} |\nabla u|^2 + |x|^{-1}|u|^4 + |x|^{-2}|u|^2 \,dx\biggr\} \label{virial-error}
\end{align}
In particular, by Lemma~\ref{energy-trapping} and Lemma~\ref{L:coercive}, we have that
\[
\eqref{virial-main} \geq \delta\int |\nabla u|^2\,dx \gtrsim \delta E(u)
\]
uniformly over $t\in[0,\infty)$ for some $\delta>0$.  On the other hand, Lemma~\ref{L:tight} yields $|\eqref{virial-error}|<\eps$ uniformly over $t\in[0,\infty)$. Thus applying the fundamental theorem of calculus and integrating over an interval of the form $[0,T]$, we derive
\[
E(u) \lesssim \tfrac{R^2}{\delta T}E(u) + \eps.
\]
Choosing $T$ sufficiently large, we obtain $E(u)\lesssim\eps$.  As $\eps$ was arbitrary, we conclude that $E(u)\equiv 0$, and hence $u\equiv 0$, as desired.\end{proof}

As in the previous section, the fact that $u$ is not identically zero therefore precludes the possibility that $T_{\max}=\infty$.  As we already know that $T_{\max}<\infty$ is impossible (from Section~\ref{S:FTBU}), we conclude that no solution as in Theorem~\ref{T:reduction} can exist, thus completing the proof of Theorem~\ref{T}.


\end{document}